\documentclass[a4paper,DIV=classic,10pt]{myart}
\KOMAoptions{DIV=last}

\usepackage[normalem]{ulem}

\newcommand{\norm}[1]{\left\Vert#1\right\Vert}
\newcommand{\abs}[1]{\left\vert#1\right\vert}
\newcommand{\Set}[1]{\ensuremath{ \left\{ #1 \right\} }}
\newcommand{\set}[1]{\ensuremath{ \{ #1 \} }}

\DeclareMathOperator*{\esssup}{ess\,sup}
\DeclareMathOperator*{\essinf}{ess\,inf}

\DeclareMathOperator*{\argmin}{arg\,min}

\usepackage{verbatim}

\DeclareMathOperator*{\aff}{aff}

\DeclareMathOperator*{\conv}{conv}
\DeclareMathOperator*{\diam}{diam}
\DeclareMathOperator*{\points}{ext}
\DeclareMathOperator*{\Sym}{S_N}

\begin{document}

\title{Brouwer Fixed Point Theorem in $\left(L^0\right)^d$}
\ArXiV{1305.2890}
\thanksColleagues{We thank Asgar Jamneshan for fruitful discussions.}

\author[a,1,s]{Samuel Drapeau}
\author[a,2,t]{Martin Karliczek}
\author[b,3,s]{Michael Kupper}
\author[a,4,v]{Martin Streckfu\ss}
\address[a]{Humboldt-Universit\"at Berlin, Unter den Linden 6, 10099 Berlin, Germany}
\address[b]{Universit\"at Konstanz, Universit\"atsstra\ss e 10, 78464 Konstanz, Germany}
\eMail[1]{drapeau@math.hu-berlin.de}
\eMail[2]{karliczm@math.hu-berlin.de}
\eMail[3]{kupper@uni-konstanz.de}
\eMail[4]{streckfu@math.hu-berlin.de}


\myThanks[s]{Funding: MATHEON project E.11}
\myThanks[t]{Funding: Konsul Karl und Dr. Gabriele Sandmann Stiftung}
\myThanks[v]{Funding: Evangelisches Studienwerk Villigst}

\date{September 12, 2013}

\abstract{
	 The classical Brouwer fixed point theorem states that in $\mathbb{R}^d$ every continuous function from a convex, compact set on itself has a fixed point.
	 For an arbitrary probability space, let $L^0=L^0(\Omega,\mathcal{A},P)$ be the set of random variables.
	 We consider $(L^0)^d$ as an $L^0$-module and show that local, sequentially continuous functions on $L^0$-convex, closed and bounded subsets have a fixed point which is measurable by construction.
}
\keyWords{Conditional Simplex, Fixed Points in $(L^0)^d$, Brouwer}
\keyAMSClassification{47H10, 13C13, 46A19, 60H25}

\maketitle

\section*{Introduction}\label{s:intro}

The Brouwer fixed point theorem states that a continuous function from a compact and convex set in $\mathbb{R}^d$ to itself has a fixed point. 
This result and its extensions play a central role in  Analysis, Optimization and Economic Theory among others. 
 To show the result one approach is to consider functions on simplexes first and use Sperner's lemma.

Recently, \citet{cheridito2012}, inspired by the theory developed by \citet{kupper03} and \citet{guo01}, studied $(L^0)^d$ as an $L^0$-module, discussing concepts like linear independence, 
$\sigma$-stability, locality and $L^0$-convexity.
Based on this, we define affine independence and conditional simplexes in $(L^0)^d$.
Showing first a result similar to  Sperner's Lemma, we obtain a fixed point for local, sequentially continuous functions on conditional simplexes.
From the measurable structure of the problem, it turns out that we have to work with local, measurable labeling functions.
To cope with this difficulty and to maintain some uniform properties, we subdivide the conditional simplex barycentrically.
We then prove the existence of a measurable completely labeled conditional simplex, contained in
the original one, which turns out to be a suitable $\sigma$-combination of
elements of the barycentric subdivision along a partition of $\Omega$.
Thus, we can construct a sequence of conditional simplexes converging to a point.
By applying always the same rule of labeling using the locality of the function, we show that this point is a fixed point.
Due to the measurability of the labeling function the fixed point is measurable by construction.
Hence, even though we follow the approach of $\mathbb{R}^d$ (cf. \cite{Border}) we do not need any measurable selection argument.

In Probabilistic Analysis theory the problem of finding random fixed points of random operators is an important issue.
Given $\mathcal{C}$, a compact convex set of a Banach space, a continuous random operator is a function $R \colon \Omega \times \mathcal{C}\to \mathcal{C}$ satisfying
\begin{enumerate}[label=\textit{(\roman*)}]
\item $R(.,x) \colon \Omega \to \mathcal{C}$ is a random variable for any fixed $x \in \mathcal{C}$,
\item $R(\omega,.) \colon \mathcal{C}\to \mathcal{C}$ is a continuous function for any fixed $\omega \in \Omega$.
\end{enumerate}
For $R$ there exists a random fixed point which is a random variable $\xi \colon \Omega \to \mathcal{C}$ such that
$\xi(\omega)=R(\omega,\xi(\omega))$ for any $\omega$ (cf. \cite{bharucha}, \cite{shahzad}, \cite{fierro}).
In contrast to this $\omega$-wise consideration, our approach is completely within the theory of $L^0$.
All objects and properties are therefore defined in that language and proofs are done with $L^0$-methods.
Moreover, the connection between continuous random operators on $\mathbb{R}^d$ and
 sequentially continuous functions on $(L^0)^d$  is not entirely clear.
 
An application, though not studied in this paper, is for instance possible in economic theory or optimization in the context of \cite{kupper08}.
Therein methods from convex analysis are used to obtain equilibrium results for translation invariant utility functionals on $(L^0)^d$.
Without translation invariance these methods  fail, and will be replaced by fixed point arguments in an ongoing work.
Thus, our result is helpful to develop the theory of non-translation invariant preference functionals mapping to $L^0$.

The present paper is organized as follows. In the first chapter we present the basic concepts concerning $(L^0)^d$ as an $L^0$-module. We define conditional simplexes and examine their basic properties.
In the second chapter we define measurable labeling functions and show the Brouwer fixed point theorem for conditional simplexes 
via a construction in the spirit of Sperner's lemma.
In the third chapter, we show a fixed point result for $L^0$-convex, bounded and sequentially closed sets in $(L^0)^d$.
With this result at hand, we present the topological implications known from the real-valued case. On the one hand, we show 
the impossibility of contracting a ball to a sphere in $(L^0)^d$ and on the other hand, an intermediate value theorem in $L^0$.

\section{Conditional Simplex} \label{sec:02}
For a probability space $(\Omega,\mathcal{A},P)$, let $L^0=L^0(\Omega,\mathcal{A},P)$ be the space of all 
$\mathcal{A}$-measurable random variables, where $P$-almost surely equal random variables are identified. 
In particular, for $X,Y \in L^0$, the relations $X\geq Y$ and $X>Y$  have to be understood $P$-almost surely. 
The set $L^0$ with the $P$-almost everywhere order is a lattice ordered ring, and every nonempty subset $\mathcal{C}\subseteq L^0$ has a least upper bound $\esssup \mathcal{C}$ 
and a greatest lower bound $\essinf \mathcal{C}$. 
For $m\in\mathbb{R}$, we denote the constant random variable $m 1_{\Omega}$ by $m$. Further, we define the sets $L^0_+=\set{X \in L^0 \colon X \geq 0}$, $L^0_{++}=\set{X \in L^0 
\colon X>0}$ and $\mathcal{A}_+=\set{A\in \mathcal{A} \colon P(A)>0}$. 
The set of random variables with values in a set $M \subseteq \mathbb{R}$ is denoted by $M(\mathcal{A})$. 
For example, $\{1,\dots,r\}(\mathcal{A})$ is the set of $\mathcal{A}$-measurable functions with values in $\{1,\dots,r\}\subseteq \mathbb{N}$, 
$[0,1](\mathcal{A})=\set{Z \in L^0 \colon 0\leq Z\leq 1}$ and $(0,1)(\mathcal{A})=\set{Z \in L^0 \colon 0< Z< 1}$.

The \emph{convex hull} of $X_1,\dots, X_N \in (L^0)^d$, $N\in \mathbb{N}$, is defined as 
\begin{align*}
    \conv\left(X_1,\dots,X_N\right)=\Set{\sum_{i=1}^N \lambda_i X_i \colon \lambda_i \in L^0_+, \sum_{i=1}^N \lambda_i=1}. 
\end{align*}
An element $Y=\sum_{i=1}^N \lambda_i X_i$ such that  $\lambda_i>0$ for all $i \in I \subseteq\{1,\dots,N\}$
is called a \emph{strict convex combination} of $\{X_i \colon i \in I\}$.

The $\sigma$-\emph{stable hull} of a set $\mathcal{C}\subseteq (L^0)^d$ is defined as
\begin{align*}
    \sigma\left(\mathcal{C}\right)=\Set{\sum_{i \in\mathbb{N}} 1_{A_i} X_i \colon X_i \in \mathcal{C}, (A_i)_{i\in\mathbb{N}} \text{ is a \emph{partition}}}, 
\end{align*}
where a partition is a countable family $(A_i)_{i \in \mathbb{N}} \subseteq \mathcal{A}$ such that $P(A_i\cap A_j)=0$ for $i\neq j$ and $P(\bigcup_{i \in \mathbb{N}} A_i)=1$. 
We call a nonempty set $\mathcal{C}$ $\sigma$-\emph{stable} if it is equal to $\sigma(\mathcal{C})$.
For a $\sigma$-stable set $\mathcal{C}\subseteq (L^0)^d$  a function $f \colon\mathcal{C} \to (L^0)^d$  is called \emph{local} 
if $f(\sum_{i\in \mathbb{N}}1_{A_i}X_i)=\sum_{i \in \mathbb{N}}1_{A_i}f(X_i)$ for every partition $(A_i)_{i \in \mathbb{N}}$ and $X_i \in \mathcal{C}$, $i\in\mathbb{N}$.
For $\mathcal{X},\mathcal{Y} \subseteq (L^0)^d$, we call a function $f \colon \mathcal{X}\to \mathcal{Y}$ \emph{sequentially continuous} if for every 
sequence $(X_n)_{n \in \mathbb{N}}$ in $\mathcal{X}$ converging to $X \in \mathcal{X}$ $P$-almost-surely 
it holds that $f(X_n)$ converges to $f(X)$ $P$-almost surely.
Further, the $L^0$\emph{-scalar product} and $L^0$\emph{-norm} on $(L^0)^d$ are defined as
\begin{equation*}
    \langle X,Y \rangle=\sum_{i=1}^d X_iY_i \quad \text{and} \quad \norm{X}=\langle X,X \rangle ^{\frac{1}{2}}.
\end{equation*}
We call $\mathcal{C}\subseteq (L^0)^d$ \emph{bounded} if $\esssup_{X \in \mathcal{C}} \norm{X} \in L^0$ and \emph{sequentially closed} if it contains all $P$-almost sure limits of sequences in $\mathcal{C}$.
Further, the diameter of $\mathcal{C}\subseteq (L^0)^d$ is defined as $\diam( \mathcal{C})=\esssup_{X,Y \in \mathcal{C}} \norm{X-Y}$.

\begin{definition} \label{d:affine}
    Elements $X_1,\ldots, X_N$ of $(L^0)^d$, $N\in\mathbb{N}$, are said to be \emph{affinely independent}, if either $N=1$ or $N>1$ and $\set{X_i-X_N}_{i=1}^{N-1}$ are \emph{linearly independent}, that is
    \begin{align}\label{affin1}
        \sum_{i=1}^{N-1} \lambda_i (X_i-X_N)=0\quad \text{implies} \quad \lambda_1=\cdots =\lambda_{N-1}=0,
    \end{align}
    where $\lambda_1,\ldots,\lambda_{N-1} \in L^0$.
\end{definition}
The definition of affine independence is equivalent to
\begin{align}\label{affin2}
    \sum_{i=1}^N\lambda_i X_i=0 \; \text{ and } \sum_{i=1}^N \lambda_i=0 \quad \text{implies}\quad \lambda_1=\cdots=\lambda_N=0.
\end{align}

Indeed, first, we show that \eqref{affin1} implies \eqref{affin2}. 
Let $\sum_{i=1}^N \lambda_i X_i=0$ and $\sum_{i=1}^N \lambda_i=0$. 
Then, $\sum_{i=1}^{N-1} \lambda_i (X_i-X_N)=\lambda_N X_N+\sum_{i=1}^{N-1} \lambda_i X_i=0$. 
By assumption \eqref{affin1}, $\lambda_1=\cdots=\lambda_{N-1}=0$, thus also $\lambda_N=0$. 
To see that \eqref{affin2} implies \eqref{affin1}, let $\sum_{i=1}^{N-1} \lambda_i (X_i-X_N)=0$. 
With $\lambda_N=-\sum_{i=1}^{N-1}\lambda_i$, it holds $\sum_{i=1}^N \lambda_i X_i=\lambda_N X_N+\sum_{i=1}^{N-1} \lambda_i X_i=\sum_{i=1}^{N-1} \lambda_i (X_i-X_N)=0$. 
By assumption \eqref{affin2}, $\lambda_1=\cdots=\lambda_{N}=0$.

\begin{remark}
    We observe that if $(X_i)_{i=1}^N\subseteq (L^0)^d$ are affinely independent then 
    $(\lambda X_i)_{i=1}^N$, for $\lambda \in L^0_{++}$, and $(X_i+Y)_{i=1}^N$, for $Y \in (L^0)^d$, are affinely independent.
    Moreover, if a family $X_1,\dots,X_N$ is affinely independent then also $1_BX_1,\dots,1_B X_N$ are affinely independent 
    on $B\in \mathcal{A}_{+}$, which means from $\sum_{i=1}^N 1_B\lambda_i X_i=0$ and $\sum_{i=1}^N 1_B\lambda_i=0$ always follows $1_B\lambda_i=0$ for all $i=1,\dots,N$. 
\end{remark}

\begin{definition}
    A \emph{conditional simplex} in $(L^0)^d$ is a set of the form
    \begin{align*}
        \mathcal{S}=\conv(X_1,\ldots,X_N)
    \end{align*}
    such that $X_1,\ldots,X_N \in (L^0)^d$ are affinely independent. 
    We call $N\in\mathbb{N}$ the dimension of $\mathcal{S}$.
\end{definition}

\begin{remark}
    In a conditional simplex $\mathcal{S}=\conv(X_1,\dots,X_N)$, the coefficients of convex combinations are unique in the sense that
    \begin{align}\label{wichtig}
        \sum_{i=1}^N \lambda_i X_i=\sum_{i=1}^N \mu_i X_i \text{ and } \sum_{i=1}^N\lambda_i=\sum_{i=1}^N \mu_i=1 \quad \text{implies} \quad \lambda_i=\mu_i \text{ for all } i=1,\ldots,N.
    \end{align}
    Indeed, since $\sum_{i=1}^N (\lambda_i-\mu_i)X_i=0$ and $\sum_{i=1}^N (\lambda_i-\mu_i)=0$, it follows from \eqref{affin2} that
    $\lambda_i-\mu_i=0$ for all $i=1,\dots,N$. 
\end{remark}

\begin{remark}\label{lpremark}
    Note that the present setting --- $L^0$-modules and the sequential $P$-almost sure convergence --- is of local nature.
    This is for instance, not the case for subsets of $L^p$ or the convergence in the $L^p$-norm for $1\leq p <\infty$.
    First, $L^p$ is not closed under multiplication and hence neither a ring nor a module over itself so that we can not even speak about affine independence.
    Second, it is mostly not a $\sigma$-stable subspace of $L^0$.
   However, for  a conditional simplex $\mathcal{S}=\conv(X_1,\ldots,X_N)$ in $(L^0)^d$ such that any $X_k$ is in $(L^p)^d$, it holds that $\mathcal{S}$ is uniformly bounded by $N \sup_{k=1,\ldots,N}\norm{X_k}\in L^p$.
    This uniform boundedness yields that any $P$-almost sure converging sequence in $\mathcal{S}$ is also converging in the $L^p$-norm for $1\leq p<\infty$ due to the dominated convergence theorem.
    This shows how one can translate results from $L^0$ to $L^p$.
\end{remark}

Since a conditional simplex is a convex hull it is in particular $\sigma$-stable.
In contrast to a simplex in $\mathbb{R}^d$ the representation of $\mathcal{S}$ as a convex hull of affinely independent elements is unique but up to $\sigma$-stability.

\begin{proposition}
    Let $(X_i)_{i=1}^N$ and $(Y_i)_{i=1}^N$ be families in $(L^0)^d$ such that $\sigma(X_1,\ldots,X_N)=\sigma(Y_1,\ldots,Y_N)$.
    Then $\conv(X_1,\dots, X_N)=\conv(Y_1,\dots,Y_N)$. 
    Moreover,  $(X_i)_{i=1}^N$ are affinely independent if and  only if $(Y_i)_{i=1}^N$ are affinely independent.

    If $\mathcal{S}$ is a conditional simplex such that $\mathcal{S}=\conv(X_1,\dots, X_N)=\conv(Y_1,\dots,Y_N)$,
    then it holds $\sigma(X_1,\dots,X_N)=\sigma(Y_1,\dots,Y_N)$.
\end{proposition}

\begin{proof}
    Suppose $\sigma(X_1,\dots,X_N)=\sigma(Y_1,\dots,Y_N)$.
    For $i=1,\dots,N$, it holds
    \begin{equation*}
        X_i \in \sigma(X_1,\dots,X_N)=\sigma(Y_1,\dots,Y_N)\subseteq \conv(Y_1,\dots, Y_N).
    \end{equation*}
    Therefore, $\conv(X_1,\dots, X_N)\subseteq \conv(Y_1,\dots,Y_N)$ and the reverse inclusion holds analogously.

    Now, let $(X_i)_{i=1}^N$ be affinely independent and $\sigma(X_1,\ldots,X_N)=\sigma(Y_1,\ldots,Y_N)$. We want to show that $(Y_i)_{i=1}^N$ are affinely independent. 
    To that end, we define the affine hull
    \begin{equation*}
        \aff(X_1,\dots,X_N)=\Set{\sum_{i=1}^N \lambda_i X_i \colon \lambda_i \in L^0, \sum_{i=1}^N\lambda_i=1}.
    \end{equation*}

    First, let $Z_1,\ldots,Z_M \in (L^0)^d$, $M\in\mathbb{N}$, such that $\sigma(X_1,\dots,X_N)=\sigma(Z_1,\dots,Z_M)$.
    We show that if $1_A \aff(X_1,\dots,X_N)\subseteq 1_A \aff(Z_1,\dots,Z_M)$ for $A\in\mathcal{A}_{+}$ and $X_1,\ldots,X_N$ are affinely independent then $M\geq N$.
    Since $X_i\in \sigma(X_1,\dots,X_N)=\sigma(Z_1,\dots, Z_M)\subseteq \aff(Z_1,\dots,Z_M)$, we have 
    $\aff(X_1,\dots,X_N)\subseteq \aff(Z_1,\dots,Z_M)$. 
    Further, it holds that  $X_1=\sum_{i=1}^M 1_{B^1_i}Z_i$ for a partition $(B^1_i)_{i=1}^M$ and hence there exists
    at least one $B^1_{k_1}$ such that $A^1_{k_1}:=B^1_{k_1} \cap A \in\mathcal{A}_{+}$, 
    and  $1_{A^1_{k_1}}X_1=1_{A^1_{k_1}} Z_{k_1}$.
    Therefore, 
    \begin{align*}
        1_{A^1_{k_1}} \aff(X_1,\dots,X_N)\subseteq 1_{A^1_{k_1}}\aff(Z_1,\dots,Z_M)=1_{A^1_{k_1}}\aff(\{X_1,Z_1,\dots,Z_M\}\setminus \{Z_{k_1}\}).
    \end{align*}
    For $X_2=\sum_{i=1}^M 1_{A^2_i}Z_i$ we find a set $A^2_k$, such that $A^2_{k_2}=A^2_k \cap A^1_{k_1}\in\mathcal{A}_{+}$,  $1_{A^2_{k_2}} X_2=1_{A_{k_2}^2} Z_{k_2}$ and $k_1\neq k_2$.
    Assume to the contrary $k_2=k_1$, then there exists a set  $B \in \mathcal{A}_+$, 
    such that $1_BX_1=1_BX_2$ which is a contradiction to the affine independence of $(X_i)_{i=1}^N$. Hence, we can again substitute $Z_{k_2}$ by $X_2$ on $A_{k_2}^2$.
    Inductively, we find  $k_1,\ldots,k_N$
    such that
    \begin{align*}
        1_{A_{k_N}} \aff(X_1,\dots,X_N)\subseteq 1_{A_{k_N}} \aff(\{X_1,\dots,X_N,Z_1,\dots,Z_M\}\setminus \{Z_{k_1},\dots Z_{k_N}\})
    \end{align*}
    which shows $M\geq N$.
    Now suppose $Y_1,\dots,Y_N$ are not affinely independent. 
    This means, there exist $(\lambda_i)_{i=1}^N$ such that $\sum_{i=1}^N \lambda_i Y_i=\sum_{i=1}^N\lambda_i=0$ 
    but not all coefficients $\lambda_i$ are zero, without loss of generality, $\lambda_1>0$ on $A\in\mathcal{A}_{+}$. 
    Thus, $1_A Y_1=-1_A\sum_{i=2}^N \frac{\lambda_i}{\lambda_1} Y_i$ and
    it holds $1_A \aff(Y_1,\dots,Y_N)=1_A \aff(Y_2,\dots,Y_N)$.  
    To see this, consider $1_AZ=1_A\sum_{i=1}^N \mu_i Y_i \in 1_A \aff(Y_1,\dots,Y_N)$ which means $1_A\sum_{i=1}^N\mu_i=1_A$.
    Thus, inserting for $Y_1$,
    \begin{align*}
        1_AZ=1_A\left[\sum_{i=2}^N \mu_i Y_i-\mu_1 \sum_{i=2}^N \frac{\lambda_i}{\lambda_1} Y_i \right]=1_A\left[\sum_{i=2}^N \left(\mu_i-\mu_1 \frac{\lambda_i}{\lambda_1}\right)Y_i\right].
    \end{align*} 
    Moreover, 
    \begin{align*}
        1_A\left[ \sum_{i=2}^N \left(\mu_i-\mu_1 \frac{\lambda_i}{\lambda_1}\right)\right]&=1_A\left[\sum_{i=2}^N\mu_i \right]+1_A\left[-\frac{\mu_1}{\lambda_1} \sum_{i=2}^N \lambda_i\right]
        = 1_A(1-\mu_1)+1_A\frac{\mu_1}{\lambda_1}\lambda_1=1_A.
    \end{align*}
    Hence, $1_AZ \in 1_A \aff(Y_2,\dots,Y_N)$.
    It follows that $1_A \aff(X_1,\dots,X_N)=1_A \aff(Y_1,\dots,Y_N)=1_A \aff(Y_2,\dots,Y_N)$. This is a contradiction to the former part of the proof (because $N-1\not\geq N$).

    Next, we characterize \emph{extremal} points in $\mathcal{S}=\conv(X_1,\dots,X_N)$.
    To this end, we show $X \in \sigma(X_1,\dots,X_N)$ if and only if 
    there do not exist  $Y$ and $Z$  in $\mathcal{S}\setminus \{X\}$  and $\lambda \in (0,1)(\mathcal{A})$ such that
    $\lambda Y +(1-\lambda) Z=X$.
    Consider $X \in \sigma(X_1,\dots,X_N)$ which is $X=\sum_{k=1}^N 1_{A_k} X_k$ for a partition $(A_k)_{k\in\mathbb{N}}$. 
    Now assume to the contrary that we find $Y=\sum_{k=1}^N \lambda_k X_k$ and $Z=\sum_{k=1}^N \mu_k X_k$ in $\mathcal{S}\setminus \{X\}$ such that 
    $X=\lambda Y+(1-\lambda)Z$.
    This means that $X=\sum_{k=1}^N (\lambda \lambda_k+(1-\lambda)\mu_k)X_k$. 
    Due to uniqueness of the coefficients (cf. \eqref{wichtig}) in a conditional simplex we have 
    $\lambda\lambda_k+(1-\lambda)\mu_k=1_{A_k}$ for all $k=1\ldots, N$. 
    By means of $0<\lambda<1$, it holds that $\lambda \lambda_k+(1-\lambda)\mu_k= 1_{A_k}$ if and only $\lambda_k=\mu_k=1_{A_k}$.
    Since the last equality holds for all $k$ it follows that $Y=Z=X$. 
    Therefore, we cannot find  $Y$ and $Z$ in $\mathcal{S}\setminus \{X\}$ such that $X$ is a strict convex combination of them.
    On the other hand, consider $X \in \mathcal{S}$ such that $X\notin \sigma(X_1,\dots,X_N)$. 
    This means, $X=\sum_{k=1}^N \nu_k X_k$, such that there exist $\nu_{k_1}$ and $\nu_{k_2}$ and $B \in \mathcal{A}_+$ with
    $0< \nu_{k_1}<1$ on $B$ and $0< \nu_{k_2}<1$ on $B$.
    Define $\varepsilon:=\essinf\set{\nu_{k_1},\nu_{k_2}, 1-\nu_{k_1},1-\nu_{k_2}}$.
    Then define $\mu_k=\lambda_k=\nu_k$ if $k_1\neq k\neq k_2$ and $\lambda_{k_1}=\nu_{k_1}-\varepsilon$, 
    $\lambda_{k_2}=\nu_{k_2}+\varepsilon$, $\mu_{k_1}=\nu_{k_1}+\varepsilon$ and $\mu_{k_2}=\nu_{k_2}-\varepsilon$.
    Thus, $Y=\sum_{k=1}^N \lambda_k X_k$ and $Z=\sum_{k=1}^N \mu_k X_k$ fulfill $0.5 Y+0.5Z=X$ but both are not equal to $X$ by construction.
    Hence, $X$ can be written as a strict convex combination of elements in $\mathcal{S}\setminus \{X\}$.
    To conclude, consider $X \in \sigma(X_1,\dots,X_N) \subseteq \mathcal{S}=\conv(X_1,\dots,X_N)=\conv(Y_1,\dots,Y_N)$. 
    Since $X \in \sigma(X_1,\dots, X_N)$ it is not a strict convex combinations of elements in $\mathcal{S}\setminus \{X\}$, in particular, of elements in $\conv(Y_1,\dots, Y_N)\setminus \{X\}$. 
    Therefore, $X$ is also in $\sigma(Y_1,\dots,Y_N)$.
    Hence, $\sigma(X_1,\dots,X_N)\subseteq \sigma(Y_1,\dots,Y_N)$. 
    With the same argumentation the other inclusion follows.
\end{proof}

As an example let us consider $[0,1](\mathcal{A})$. 
For an arbitrary  $A \in \mathcal{A}$, it holds that
$1_A$ and $1_{A^c}$ are affinely independent and  $\conv(1_A,1_{A^c})=\set{\lambda 1_A+ (1-\lambda)1_{A^c}\colon 0\leq \lambda \leq  1}=[0,1](\mathcal{A})$.
Thus, the conditional simplex $[0,1](\mathcal{A})$ can be written as a convex combination of different affinely independent elements of $L^0$.
This is due to the fact that $\sigma(0,1)=\set{1_B: B \in \mathcal{A}}=\sigma(1_A,1_{A^c})$ for all $A\in \mathcal{A}$.

\begin{remark}
    In $(L^0)^d$, let  $e_i$ be the random variable which is $1$ in the $i$-th component and $0$ in any other.
    Then the family $0,e_1,\ldots,e_d$ is affinely independent and 
    $(L^0)^d=\aff(0,e_1,\ldots,e_d)$.
    Hence, the maximal number of affinely independent elements in $(L^0)^d$ is $d+1$.
\end{remark}

The characterization of $X\in \sigma(X_1,\dots,X_N)$ leads to the following definition.
\begin{definition}
    Let $\mathcal{S}=\conv(X_1,\dots,X_N)$ be a conditional simplex.
    We define the set of \emph{extremal points} $\points(\mathcal{S})=\sigma(X_1,\dots,X_N)$.
    For an index set $I$ and a collection $\mathscr{S}=(\mathcal{S}_i)_{i \in I}$ of conditional simplexes 
    we denote $\points(\mathscr{S})=\sigma(\points(\mathcal{S}_i) \colon i \in I)$.
\end{definition}

\begin{remark}\label{sigma}
    Let $\mathcal{S}^j=\conv(X^j_1,\dots, X^j_N)$, $j \in \mathbb{N}$, be conditional simplexes of the same dimension $N$ and
    $(A_j)_{j\in\mathbb{N}}$ a partition. 
    Then $\sum_{j\in \mathbb{N}} 1_{A_j} \mathcal{S}^j$ is again a conditional simplex. 
    To that end, we define $Y_k=\sum_{j\in \mathbb{N}} 1_{A_j} X^j_k$ and recognize $\sum_{j\in \mathbb{N}} 1_{A_j} \mathcal{S}^j=\conv(Y_1,\dots, Y_N)$.
    Indeed, 
    \begin{equation}\label{affine}
        \sum_{k=1}^N \lambda_k Y_k=\sum_{k=1}^N \lambda_k \sum_{j\in \mathbb{N}} 1_{A_j}X^j_k=\sum_{j\in \mathbb{N}}1_{A_j} \sum_{k=1}^N
        \lambda_k X^j_k \in \sum_{j\in \mathbb{N}} 1_{A_j} \mathcal{S}^j,
    \end{equation}
    shows $\conv(Y_1,\dots, Y_N)\subseteq\sum_{j\in \mathbb{N}} 1_{A_j} \mathcal{S}^j$.
    Considering $\sum_{k=1}^N \lambda_k^j X^j_k$ in $\mathcal{S}^j$ and defining $\lambda_k=\sum_{j\in\mathbb{N}} 1_{A_j} \lambda_k^j$ yields the other inclusion.
    To show that $Y_1,\dots, Y_N$ are affinely independent consider $\sum_{k=1}^N \lambda_k Y_k=0=\sum_{k=1}^N \lambda_k$.
    Then by \eqref{affine}, it holds $1_{A_j}\sum_{k=1}^N \lambda_k X^j_k=0$ and since $\mathcal{S}^j$ is a conditional simplex, 
    $1_{A_j}\lambda_k=0$ for all $j\in \mathbb{N}$ and $k=1,\dots, N$.
    From the fact that $(A_j)_{j\in \mathbb{N}}$ is a partition, it follows that $\lambda_k=0$ for all $k=1,\dots, N$.
\end{remark}

We will prove the Brouwer fixed point theorem in the present setting using an analogue version of Sperner's lemma. 
As in the unconditional case we have to subdivide a conditional simplex in smaller ones. 
For our argumentation we cannot use arbitrary subdivisions and need very special properties of the conditional simplexes in which we subdivide. 
This leads to the following definition.

\begin{definition}\label{baryz}
    Let $\mathcal{S}=\conv(X_1,\dots,X_N)$ be a conditional simplex and $\Sym$ the group of permutations of $\{1,\dots,N\}$.
    Then for $\pi \in \Sym$ we define 
    \begin{align*}
        \mathcal{C}_{\pi}=\conv\left(X_{\pi(1)},\frac{X_{\pi(1)}+X_{\pi(2)}}{2},\dots, \frac{X_{\pi(1)}+\cdots+X_{\pi(k)}}{k},\dots,\frac{X_{\pi(1)}+\cdots+X_{\pi(N)}}{N} \right).
    \end{align*}
    We call $(\mathcal{C}_{\pi})_{\pi \in \Sym}$ the \emph{barycentric subdivision} of $\mathcal{S}$, and denote $Y_k^{\pi}=\frac{1}{k}\sum_{i=1}^k X_{\pi(i)}$.
\end{definition}

\begin{lemma}\label{subdivision}
    The barycentric subdivision is a collection of finitely many conditional simplexes satisfying the following properties
    \begin{enumerate}[label=\textnormal{(\roman*)}]
        \item \label{cond01} 
            $\sigma(\bigcup_{\pi \in \Sym} \mathcal{C}_{\pi})=\mathcal{S}$. 
        \item\label{cond02}  
            $\mathcal{C}_{\pi}$  has dimension $N$, $\pi\in\Sym$.
        \item\label{cond03} 
            $\mathcal{C}_{\pi}\cap \mathcal{C}_{\overline{\pi}}$  is a conditional simplex of dimension $r\in \mathbb{N}$ and $r<N$ for $\pi,\overline{\pi}\in\Sym$, $\pi\neq \overline{\pi}$.
        \item\label{cond04} 
            For $s=1,\dots, N-1$, let $\mathcal{B}_s:=\conv(X_1,\dots, X_{s})$. All conditional simplexes $\mathcal{C}_{\pi}\cap \mathcal{B}_s$, $\pi\in\Sym$, of dimension\,$s$ subdivide $\mathcal{B}_s$ barycentrically.
    \end{enumerate} 
\end{lemma}

\begin{proof}
    We show the affine independence of $Y_1^{\pi},\ldots,Y_N^{\pi}$ in $\mathcal{C}_{\pi}$.
    It holds 
    \begin{equation*}
        \lambda_{\pi(1)} X_{\pi(1)}+\lambda_{\pi(2)} \frac{X_{\pi(1)}+X_{\pi(2)}}{2}+\cdots+\lambda_{\pi(N)}\frac{\sum_{k=1}^N X_{\pi(k)}}{N}=\sum_{i=1}^N \mu_i X_i,
    \end{equation*}
    with $\mu_{i}=\sum_{k=\pi^{-1}(i)}^N \frac{\lambda_{\pi(k)}}{k}$.
    Since $\sum_{i=1}^N \mu_i=\sum_{i=1}^N \lambda_i$, the affine independence of $Y_1^{\pi},\ldots,Y_N^{\pi}$ is obtained by the affine independence of $X_1,\ldots, X_N$.
    Therefore all $\mathcal{C}_{\pi}$ are conditional simplexes.

    The intersection of two conditional simplexes $\mathcal{C}_{\pi}$ and $\mathcal{C}_{\overline{\pi}}$ can be expressed in the following manner.
    Let $J=\set{j \colon \{\pi(1),\dots,\pi(j)\}=\{\overline{\pi}(1),\dots,\overline{\pi}(j)\}}$ be the set of indexes up to which both $\pi$ and $\overline{\pi}$ 
    have the same set of images.
    Then, 
    \begin{equation}\label{nervigerschrott}
        \mathcal{C}_{\pi}\cap \mathcal{C}_{\overline{\pi}}=\conv\left( \frac{\sum_{k=1}^j X_{\pi(k)}}{j} \colon j\in J\right).
    \end{equation}

    To show $\supseteq$,	let $j \in J$. 
    It holds that $\frac{\sum_{k=1}^j X_{\pi(k)}}{j}$ is in both $\mathcal{C}_{\pi}$ and $\mathcal{C}_{\overline{\pi}}$ since
    $\set{\pi(1),\dots, \pi(j)}=\set{\overline{\pi}(1),\dots, \overline{\pi}(j)}$. 
    Since the intersection of convex sets is convex, we get this implication.

    For the reverse inclusion, let $X \in \mathcal{C}_\pi \cap \mathcal{C}_{\overline{\pi}}$.
    From $X \in \mathcal{C}_{\pi}\cap C_{\bar{\pi}}$, it follows that $X=\sum_{i=1}^N \lambda_i (\sum_{k=1}^i \frac{X_{\pi(k)}}{i})=\sum_{i=1}^N \mu_i (\sum_{k=1}^i \frac{X_{\overline{\pi}(k)}}{i})$.
    Consider $j\not\in J$. By definition of $J$, there exist $p,q\leq j$ with
    $\overline{\pi}^{-1}(\pi(p)),$ $\pi^{-1}(\overline{\pi}(q))\not\in \{1,\ldots,j\}$. By \eqref{wichtig}, the coefficients of $X_{\pi(p)}$ are
    equal:
    $\sum_{i=p}^N\frac{\lambda_i}{i}=\sum_{i=\overline{\pi}^{-1}(\pi(p))}^N\frac{\mu_i}{i}$.
    The same holds for $X_{\pi(q)}$:
    $\sum_{i=q}^N\frac{\mu_i}{i}=\sum_{i=\pi^{-1}(\overline{\pi}(q))}^N\frac{\lambda_i}{i}$.
    Put together
    \begin{align*}
        \sum_{i=j+1}^N\frac{\mu_i}{i}\leq\sum_{i=q}^N\frac{\mu_i}{i}=\sum_{i=\pi^{-1}(\overline{\pi}(q))}^N\frac{\lambda_i}{i}\leq
        \sum_{i=j+1}^N\frac{\lambda_i}{i}\leq
        \sum_{i=p}^N\frac{\lambda_i}{i}=\sum_{i=\overline{\pi}^{-1}(\pi(p))}^N\frac{\mu_i}{i}\leq
        \sum_{i=j+1}^N\frac{\mu_i}{i}
    \end{align*}
    which is only possible if $\mu_j=\lambda_j=0$ since $p,q\leq j$.

    Furthermore, if  $\mathcal{C}_{\pi}\cap \mathcal{C}_{\overline{\pi}}$ is of dimension $N$ by \eqref{nervigerschrott}
    follows that $\pi=\overline{\pi}$.
    This shows \ref{cond03}.

    As for Condition \ref{cond01}, it clearly holds $\sigma(\cup_{\pi \in \Sym} \mathcal{C}_{\pi})\subseteq \mathcal{S}$.
    On the other hand, let $X=\sum_{i=1}^N \lambda_i X_i \in \mathcal{S}$. 
    Then, we find a partition $(A_n)_{n\in\mathbb{N}}$ such that on every $A_n$ the indexes are completely ordered which is 
    $\lambda_{i^n_1}\geq\lambda_{i^n_2}\geq\cdots\geq\lambda_{i^n_N}$ on $A_n$.
    This means, that $X \in 1_{A_n} \mathcal{C}_{\pi^n}$ with $\pi^n(j)=i^n_j$.
    Indeed, we can rewrite $X$ on $A_n$ as 
    \begin{align*}
        X=(\lambda_{i_1^n}-\lambda_{i_2^n}) X_{i_1^n}+\cdots+ (N-1)(\lambda_{i_{N-1}^n}-\lambda_{i_N^n})
        \frac{\sum_{k=1}^{N-1}X_{i_k^n}}{N-1}+N\lambda_{i_N^n} \frac{\sum_{k=1}^N X_{i_k^n}}{N},
    \end{align*}
    which shows that $X \in \mathcal{C}_{\pi^n}$ on $A_n$.

    Further, for $\mathcal{B}_{s}=\conv(X_1,\dots,X_{s})$ the elements $\mathcal{C}_{\pi^\prime}\cap \mathcal{B}_{s}$ of 
    dimension $s$ are 
    exactly the ones with $\{\pi^\prime(i) \colon i=1,\dots,s\}= \{1,\dots,s\}$. 
    To this end, let $C_{\pi^\prime} \cap \mathcal{B}_s$ be of dimension $s$.
    This means there exists an element $Y$ in this intersection such that $Y=\sum_{i=1}^N \lambda_i X_i$ with $\lambda_i>0$ for all
    $i=1,\dots, s$ and $\lambda_i=0$ for $i>s$.
    As an element of $C_{\pi^\prime}$ this $Y$ has a representation of the form 
    $Y=\sum_{j=1}^N (\sum_{k=j}^N \frac{\mu_k}{k})X_{\pi^\prime(j)}$, for $\sum_{k=1}^N \mu_k=1$ and $\mu_k\in L^0_+$ for every $k=1,\dots,N$. 
    Suppose now that there exists some $j_0\leq s$ with $\pi^\prime(j_0)>s$. 
    Then due to $\lambda_{\pi^\prime(j_0)}=0$ and the uniqueness of the coefficients (cf.\eqref{wichtig})  in an $L^0$-simplex, it holds 
    $\sum_{k=j_0}^N \frac{\mu_k}{k}=0$ and within $\sum_{k=j}^N \frac{\mu_k}{k}=0$ for all $j\geq j_0$.
    This means  $Y=\sum_{j=1}^{j_0-1} (\sum_{k=j}^N \frac{\mu_k}{k})X_{\pi^\prime(j)}$ and hence
    $Y$ is the convex combination of $j_0-1$ elements with $j_0-1<s$. 
    This contradicts the property that $\lambda_i>0$ for $s$ elements.
    Therefore, $(\mathcal{C}_{\pi^\prime}\cap \mathcal{B}_{s})_{\pi^\prime}$ is exactly the barycentric subdivision of $\mathcal{B}_{s}$, which has been shown to fulfill the properties \ref{cond01}-\ref{cond03}.
\end{proof}

Subdividing a conditional simplex $\mathcal{S}=\conv (X_1,\ldots,X_N)$ barycentrically we obtain  $(\mathcal{C}_{\pi})_{\pi \in S_N}$. 
Dividing every $\mathcal{C}_{\pi}$  barycentrically results in a new collection of conditional simplexes and we call this  the two-fold barycentric subdivision of $\mathcal{S}$.
Inductively, we can subdivide every conditional simplex of the $(m-1)$th step  barycentrically and call the resulting collection of  conditional simplexes  the
$m$-fold barycentric subdivision of $\mathcal{S}$ and denote it by $\mathscr{S}^m$.
Further, we define $\points(\mathscr{S}^m)=\sigma( \{\points(\mathcal{C})\colon \mathcal{C} \in \mathscr{S}^m\})$ to be the $\sigma$-stable hull of all extremal points of the conditional 
simplexes of the $m$-fold barycentric subdivision of $\mathcal{S}$. Notice that this is the $\sigma$-stable hull of only finitely many elements, since there are only finitely many simplexes 
in the subdivision each of which is the convex hull of $N$ elements.
 \begin{remark}\label{folge}
     Consider an arbitrary $\mathcal{C}_{\pi}=\conv(Y_1^{\pi},\ldots,Y_N^{\pi})$, $\pi\in\Sym$ in the barycentric subdivision of a conditional simplex $\mathcal{S}$.
     Then it holds
     \begin{align*}
         \diam(\mathcal{C}_{\pi}) \leq \esssup_{i,j=1,\ldots,N} \norm{Y_i^{\pi}-Y_j^{\pi}} \leq \frac{N-1}{N} \diam(\mathcal{S}).
     \end{align*}
     Since this holds for any $\pi\in S_N$, it follows that the diameter of $\mathcal{S}^m$, which is an arbitrary conditional simplex of the $m$-fold
     barycentric subdivision of $\mathcal{S}$, fulfills  $\diam(\mathcal{S}^m)\leq \left(\frac{N-1}{N}\right)^m \diam(\mathcal{S})$. 
     Since $\diam(\mathcal{S})<\infty$ and $\left(\frac{N-1}{N}\right)^m\rightarrow 0$, for $m\rightarrow \infty$, it follows that
     $\diam(\mathcal{S}^m)\rightarrow 0$, for $m\rightarrow \infty$ for every sequence $(\mathcal{S}^m)_{m \in \mathbb{N}}$.
 \end{remark}

\section{Brouwer Fixed Point Theorem for Conditional Simplexes}
\begin{definition} \label{ecken}
	Let $\mathcal{S}=\conv(X_1,\dots, X_N)$ be a conditional simplex, $m$-fold barycentrically subdivided in $\mathscr{S}^m$.
	A local function $\phi \colon \points(\mathscr{S}^m)  \to \{1,\dots,N\}(\mathcal{A})$ is called a \emph{labeling function} of $\mathcal{S}$.
	For fixed $X_1,\ldots,X_N\in \points({\mathcal{S}})$ with $\mathcal{S}=\conv(X_1,\dots, X_N)$, 
	the labeling function is called \emph{proper}, if for any $Y \in \points(\mathscr{S}^m)$ it holds that
	\begin{align*}
		P\left(\left\{\phi(Y)=i\right\}\subseteq \left\{\lambda_i>0\right\}\right)=1,
	\end{align*} 
	for $i=1,\dots,N$, where $Y=\sum_{i=1}^N\lambda_i X_i$.
	A conditional simplex $\mathcal{C}=\conv(Y_1,\dots,Y_N)\subseteq \mathcal{S}$, with $Y_j\in \points(\mathscr{S}^m), j=1,\dots,N$, is said to be \emph{completely labeled} by 
	$\phi$ if this is a proper labeling function of $\mathcal{S}$ and 
	\begin{equation*}
		P\left(\bigcup_{j=1}^N \{ \phi(Y_j)=i \}\right)=1,
	\end{equation*}
	for all $i \in \{1,\dots,N\}$.
\end{definition}

\begin{lemma}\label{label}
	Let $\mathcal{S}=\conv(X_1,\ldots,X_N)$ be a conditional simplex and $f \colon \mathcal{S}\to \mathcal{S}$ be a local function.  
	Let $\phi \colon \points(\mathscr{S}^m)  \to \{0,\dots,N\}(\mathcal{A})$ be a local function such that
	\begin{enumerate}[label=\textnormal{(\roman*)}]
		\item $P\left(\set{\phi(X)=i}\subseteq \set{\lambda_i>0} \cap \set{\lambda_i\geq \mu_i}\right)=1$, for all $i=1,\dots,N$,\label{dort}
		\item $P\left(\bigcup_{i=1}^N \big(\set{\lambda_i>0} \cap \set{\lambda_i\geq \mu_i}\big) \subseteq \bigcup_{i=1}^N \set{\phi(X)=i}\right)=1$, \label{nichtnull}
	\end{enumerate}
	where $X=\sum_{i=1}^N \lambda_i X_i$ and $f(X)=\sum_{i=1}^N\mu_iX_i$.
	Then, $\phi$ is a proper labeling function. 

	Moreover, the set of functions fulfilling these properties is non-empty.
\end{lemma}

\begin{proof}
	First we show that $\phi$ is a labeling function. Since $\phi$ is local we just have to prove that $\phi$ actually maps to $\{1,\dots,N\}$. 
	Due to \ref{nichtnull}, we have to show that $P(\bigcup_{i=1}^N\left\{\lambda_i\geq \mu_i \colon \lambda_i >0\right\})=1$. Assume to
	the contrary, $\mu_i > \lambda_i$ on $A\in\mathcal{A}_{+}$, for all $\lambda_i$ with $\lambda_i>0$ on $A$. 
	Then it holds that  $1=\sum_{i=1}^N
	\lambda_i 1_{\{\lambda_i>0\}}<\sum_{i=1}^N \mu_i 1_{\{\mu_i>0\}}=1$ on $A$ which yields a
	contradiction. Thus, $\phi$ is a labeling function.
	Moreover, due to \ref{dort} it holds that $P\left(\{\phi(X)=i\}\subseteq \{\lambda_i>0\}\right)=1$ which shows that $\phi$ is proper.

	To prove the existence, for $X\in \points(\mathscr{S}^m)$ with $X=\sum_{i=1}^N \lambda_i X_i$, $f(X)=\sum_{i=1}^N \mu_i$ 
	let  $B_i:=\{\lambda_i>0\} \cap \{\lambda_i\geq \mu_i\}$, $i=1,\dots,N$.
	Then we define the function $\phi$ at $X$ as $\{\phi(X)=i\}=B_i \setminus (\bigcup_{k=1}^{i-1} B_k)$, $i=1,\dots,N$.
	It has been shown that $\phi$ maps to $\{1,\dots,N\}(\mathcal{A})$ and is proper.
	It remains to show that $\phi$ is local. 
	To this end, consider $X=\sum_{j \in \mathbb{N}} 1_{A_j} X^j$ where $X^j=\sum_{i=1}^N \lambda^j_i X_i$ and $f(X^j)=\sum_{i=1}^N \mu^j_i X_i$.
	Due to uniqueness of the coefficients in a conditional simplex it holds that $\lambda_i=\sum_{j\in \mathbb{N}} 1_{A_j} \lambda^j_i$ 
	and due to locality
	of $f$ it follows that $\mu_i=\sum_{j\in \mathbb{N}} 1_{A_j} \mu^j_i$.
	Therefore it holds that 
	$B_i=\bigcup_{j \in \mathbb{N}} \left(\{\lambda^j_i>0\}\cap
	\{\lambda^j_i \geq \mu^j_i\}\cap A_j\right)=\bigcup_{j \in \mathbb{N}} (B^j_i\cap A_j)$.
	Hence, $\phi(X)=i$ on $B_i\setminus (\bigcup_{k=1}^{i-1} B_k)= 
	[\bigcup_{j \in \mathbb{N}} (B_i^j \cap A_j)]\setminus [\bigcup_{k=1}^{i-1}(\bigcup_{j\in \mathbb{N}} B_k^j \cap A_j)]=
	\bigcup_{j \in \mathbb{N}} [(B^j_i\setminus \bigcup_{k=1}^{i-1} B^j_k)\cap A_j]$.
	On the other hand, we see that $\sum_{j \in \mathbb{N}} 1_{A_j} \phi(X^j)$ is $i$ on any  $A_j \cap \{\phi(X^j)=i\}$, 
	hence it is $i$ on $\bigcup_{j \in \mathbb{N}} (B^j_i\setminus \bigcup_{k=1}^{i-1} B^j_k) \cap A_j$.
	Thus, $\sum_{j\in \mathbb{N}} 1_{A_j}\phi(X^j)=\phi(\sum_{j\in \mathbb{N}} 1_{A_j} X^j)$ which shows that $\phi$ is local. 
\end{proof}

The reason to demand locality of a labeling function is exactly because we want to label by the rule explained in  Lemma \ref{label} and hence keep local information with it. 
For example consider a conditional simplex $\mathcal{S}=\conv(X_1,X_2,X_3,X_4)$  and $\Omega=\set{\omega_1,\omega_2}$.
Let $Y \in \points(\mathscr{S})$ be given by $Y=\frac{1}{3}\sum_{i=1}^3 X_i$.
Now consider a function $f$ on $\mathcal{S}$ such that
\begin{align*}
	f(Y)(\omega_1)=\frac{1}{4} X_1(\omega_1)+ \frac{3}{4} X_3(\omega_1); \quad
	f(Y)(\omega_2)=\frac{2}{5} X_1(\omega_2)+\frac{2}{5} X_2(\omega_2)+ \frac{1}{5} X_4(\omega_2).
\end{align*}
If we  label $Y$ by the rule explained in Lemma \ref{label},
$\phi$ takes the  values $\phi(Y)(\omega_1) \in \{1,2\}$ and $\phi(Y)(\omega_2)= 3$. 
Therefore, we can really distinguish on which sets $\lambda_i\geq \mu_i$. 
Yet, using a deterministic labeling of $Y$, we would loose this information. 

\begin{theorem}\label{thm:brower}
	Let $\mathcal{S}=\conv(X_1,\ldots,X_N)$ be a conditional simplex in $(L^0)^d$. Let $f \colon \mathcal{S} \to \mathcal{S}$ be a local,
	sequentially continuous function. Then there exists $Y\in \mathcal{S}$ such that $f(Y)=Y$.
\end{theorem}

\begin{proof}
	We consider the barycentric subdivision $(\mathcal{C}_{\pi})_{\pi\in\Sym}$ of $\mathcal{S}$ and a proper labeling function $\phi$ on 
	$\points(\mathscr{S})$.
	First, we show that we can find a completely
	labeled conditional simplex in $\mathcal{S}$.
	By induction on the dimension of $\mathcal{S}=\conv(X_1,\ldots, X_N)$, we
	show that there exists a partition $(A_k)_{k=1,\ldots,K}$ such that on any $A_k$ there is an odd number
	of completely labeled $\mathcal{C}_{\pi}$. The case $N=1$ is clear, since a
	point can  be labeled with the constant index $1$, only.

	Suppose the case $N-1$ is proven.
	Since the number of $Y_i^{\pi}$ of the barycentric subdivision is finite and $\phi$ can only take finitely many values, it holds for all
	$V\in (Y_i^{\pi})_{i=1,\ldots,N,\pi\in\Sym}$ there exists a partition $(A_k^V)_{k=1,\ldots,K}$, $K<\infty$, where $\phi(V)$ is constant on any
	$A_k^V$. Therefore, we find a partition $(A_k)_{k=1,\ldots,K}$, such that
	$\phi(V)$ on $A_k$ is constant for all $V$ and $A_k$.
	Fix $A_k$ now.

	In the following, we denote by  $\mathcal{C}_{\pi^b}$  these conditional simplexes for which $\mathcal{C}_{\pi^b}\cap \mathcal{B}_{N-1}$ are $N-1$-dimensional 
	(cf. Lemma \ref{subdivision} \ref{cond04}), therefore $\pi^b(N)=N$. 
	Further we denote by $\mathcal{C}_{\pi^c}$ these conditional simplexes  which are not of the type $\mathcal{C}_{\pi^b}$, that is $\pi^c(N)\neq N$.
	If we use $\mathcal{C}_\pi$ we mean a conditional simplex of arbitrary type. 
	We define
	\begin{itemize}
		\item   $\mathscr{C}\subseteq (\mathcal{C}_{\pi})_{\pi\in\Sym}$ to be the set of $\mathcal{C}_{\pi}$  which are completely labeled on $A_k$. 
		\item $\mathscr{A}\subseteq (\mathcal{C}_{\pi})_{\pi\in\Sym}$ to be the set of the $P$-almost completely labeled $\mathcal{C}_{\pi}$, that is
			$$\left\{ \phi (Y_k^{\pi}), k=1,\dots, N\right\}=\left\{1,\ldots,N-1\right\}\quad \text{on }A_k.$$
		\item  $\mathscr{E}_{\pi}$  to be the set of the  intersections $(\mathcal{C}_{\pi} \cap \mathcal{C}_{\pi_l})_{\pi_l\in\Sym}$ which are $N-1$-dimensional and completely labeled on $A_k$.\footnote{ That is bearing exactly the label $1,\dots,N-1$ on $A_k$.}
		\item $\mathscr{B}_{\pi}$ to be the set of the intersections $\mathcal{C}_{\pi}\cap \mathcal{B}_{N-1}$ which are completely labeled on $A_k$.
	\end{itemize}
	It holds that $\mathscr{E}_\pi \cap \mathscr{B}_\pi=\emptyset$ and hence $\abs{\mathscr{E}_\pi \cup \mathscr{B}_\pi}=\abs{\mathscr{E}_\pi}+ \abs{\mathscr{B}_\pi}$.
	Since $\mathcal{C}_{\pi^c} \cap \mathcal{B}_{N-1}$ is at most $N-2$-dimensional, it holds that $\mathscr{B}_{\pi^c}=\emptyset$ and hence $ \abs{\mathscr{B}_{\pi^c}}=0$.
	Moreover, we know that $\mathcal{C}_{\pi} \cap \mathcal{C}_{\pi_l}$ is $N-1$-dimensional on $A_k$ if and only if this holds on whole $\Omega$ 
	(cf. Lemma \ref{subdivision} \ref{cond03}) and 
	$\mathcal{C}_{\pi^b} \cap \mathcal{B}_{N-1} \neq \emptyset$ on $A_k$ if and only if 
	this also holds on whole $\Omega$ (cf. Lemma \ref{subdivision} \ref{cond04}). 
	So it does not play any role if we look at these sets which are intersections on $A_k$ or on $\Omega$ since they 
	are exactly the same sets.

	If $\mathcal{C}_{\pi^c}\in \mathscr{C}$ then $\abs{\mathscr{E}_{\pi^c}}=1$ and if $\mathcal{C}_{\pi^b} \in \mathscr{C}$ then 
	$\abs{\mathscr{E}_{\pi^b}\cup \mathscr{B}_{\pi^b}}=1$.
	If $\mathcal{C}_{\pi^c} \in \mathscr{A}$ then $\abs{\mathscr{E}_{\pi^c}}=2$ and if $\mathcal{C}_{\pi^b} \in \mathscr{A}$ 
	then $\abs{\mathscr{E}_{\pi^b}\cup \mathscr{B}_{\pi^b}}=2$.
	Therefore it holds $\sum_{\pi\in\Sym} \abs{\mathscr{E}_{\pi} \cup \mathscr{B}_{\pi}}=\abs{\mathscr{C}}+2\abs{\mathscr{A}}$.

	If we pick an $E_{\pi} \in \mathscr{E}_{\pi}$ we know there always exists exactly one other $\pi_l$ such that $E_{\pi} \in \mathscr{E}_{\pi_l}$
	(Lemma \ref{subdivision}\ref{cond03}).
	Therefore $\sum_{\pi\in \Sym} \abs{\mathscr{E}_{\pi}}$ is even. 
	Moreover $(\mathcal{C}_{\pi^b}\cap \mathcal{B}_{N-1})_{\pi^b}$ subdivides $\mathcal{B}_{N-1}$ barycentrically\footnote{
		The boundary of $\mathcal{S}$ is a $\sigma$-stable set so if it is partitioned by the
		labeling function into $A_k$ we know that $\mathcal{B}_{N-1}(\mathcal{S})=\sum_{k=1}^K 1_{A_k}\mathcal{B}_{N-1}(1_{A_k}\mathcal{S})$ and by
		Lemma \ref{subdivision} \ref{cond04}  we can apply the induction hypothesis
	also on $A_k$.}
	and hence we can apply the hypothesis (on $\points(\mathcal{C}_{\pi^b} \cap \mathcal{B}_{N-1})$). 
	This means that the number of completely labeled conditional simplexes is odd on a partition of $\Omega$ but since $\phi$ is constant on $A_k$ it also has to be odd there.
	This  means that $\sum_{\pi^b} \abs{\mathscr{B}_{\pi^b}}$ has to be odd. 
	Hence, we also have that  $\sum_{\pi} \abs{\mathscr{E}_{\pi} \cup \mathscr{B}_{\pi}}$ is the sum of an even and an odd number and thus odd. So we conclude 
	$\abs{\mathscr{C}}+2\abs{\mathscr{A}}$ is odd and 
	hence also $\abs{\mathscr{C}}$.
	Thus, we find for any $A_k$ a completely labeled $\mathcal{C}_{\pi_k}$.

	We define $\mathcal{S}^1 =\sum_{k=1}^K 1_{A_k}\mathcal{C}_{\pi_k}$ which by Remark \ref{sigma} is indeed a conditional simplex.
	Due to $\sigma$-stability of $\mathcal{S}$ it holds  $\mathcal{S}^1 \subseteq \mathcal{S}$.
	By Remark \ref{folge}  $\mathcal{S}^1$ has
	a diameter which is less then $\frac{N-1}{N} \diam(\mathcal{S})$  and since $\phi$ is local $\mathcal{S}^1$ is completely labeled on whole $\Omega$.

	The same argumentation holds for every $m$-fold barycentric subdivision $\mathscr{S}^m$ of $\mathcal{S}$, $m\in \mathbb{N}$, that is, there exists a completely labeled conditional simplex in every $m$-fold 
	barycentrically 
	subdivided conditional simplex which is properly labeled. 
	Henceforth, subdividing $\mathcal{S}$ $m$-fold barycentrically and label it by $\phi^m\colon \points(\mathscr{S}^m) \to \set{1,\dots,N}(\mathcal{A})$, which is a labeling 
	function as in Lemma \ref{label}, we always obtain a completely labeled conditional simplex $\mathcal{S}^{m+1} \subseteq \mathcal{S}$, for $m\in \mathbb{N}$.
	Moreover, since $\mathcal{S}^1$ is completely labeled, it holds $\mathcal{S}^1 =\sum_{k=1}^K 1_{A_k}\mathcal{C}_{\pi_k}$ as above where $\mathcal{C}_{\pi_k}$ is completely labeled on $A_k$. 
 	This means $\mathcal{C}_{\pi_k}=\conv(Y^k_1,\dots,Y^k_N)$ with $\phi(Y^k_j)=j$ on $A_k$ for every $j=1,\dots,N$.
 	Defining $V^1_j=\sum_{k=1}^K 1_{A_k} Y^k_j$ for every $j=1,\dots,N$ yields $P(\{\phi(V^1_j)=j\})=1$ for every $j=1,\dots,N$ and $\mathcal{S}^1=\conv(V^1_1,\dots,V_N^1)$. 
 	The same holds for any $m\in \mathbb{N}$ and so that we can write $\mathcal{S}^m=\conv(V^m_1,\dots, V^m_N)$ with $P(\phi^{m-1}(\{V^m_j)=j\})=1$ for every $j=1,\dots,N$.
    
    Now, $(V_1^m)_{m \in \mathbb{N}}$ is a sequence in the sequentially closed, $L^0$-bounded set $\mathcal{S}$, so that by \cite[Corollary 3.9]{cheridito2012},
    there exists $Y \in \mathcal{S}$ and a sequence $(M_m)_{m \in \mathbb{N}}$ in $\mathbb{N}(\mathcal{A})$ such that $M_{m+1}>M_{m}$ for all $m\in \mathbb{N}$ and 
    $\lim_{m \to \infty} V^{M_m}_1 =Y$ $P$-almost surely. For $M_m \in \mathbb{N}(\mathcal{A})$, 
    $V^{M_m}_1$ is defined as $\sum_{n \in \mathbb{N}} 1_{ \{M_m=n\}} V^n_1$. 
    This means an element with index $M_m$, for some $m\in \mathbb{N}$, equals $V^n_1$ on $A_n$, $n \in \mathbb{N}$, where the sets $A_n$ are determined by $M_m$ 
    via $A_n=\{M_m=n\}$, $n\in \mathbb{N}$.  
    Furthermore, as $m$ goes to $\infty$, $\diam(S^m)$ is converging to zero $P$-almost surely, and therefore it also  follows  that $\lim_{m\to \infty} V^{M_m}_k=Y$ $P$-almost surely for every $k=1,\ldots,N$.
    Indeed, it holds $\abs{V^m_k-Y}\leq diam(\mathcal{S}^m)+\abs{V^m_1-Y}$ for every $k=1,\dots,N$ and $m\in \mathbb{N}$ so we can use the sequence $(M_m)_{m \in \mathbb{N}}$ for every $k=1,\dots,N$.
    
    Let $Y=\sum_{l=1}^N \alpha_l X_l$ and $f(Y)=\sum_{l=1}^N \beta_l X_l$ as well as $V^{m}_k = \sum_{l=1}^N \lambda_l^{m,k} X_l$ and $f(V^{m}_k)=\sum_{l=1}^N \mu_l^{m,k}X_l$ for $m \in \mathbb{N}$.
   As $f$ is local it holds that $f(V^{M_m}_1)=\sum_{n\in \mathbb{N}} 1_{\{M_m=n\}}f(V^n_1)$.
     By sequential continuity  of $f$, it follows that $\lim_{m\to \infty} f(V^{M_n}_k)=f(Y)$ $P$-almost surely for every $k=1,\ldots,N$.
    In particular, $\lim_{m\to \infty}\lambda^{M_m,l}_l=\alpha_l$ and $\lim_{m \to \infty}\mu^{M_m,l}_l = \beta_l$ $P$-almost surely for every $l=1,\ldots,N$.
    However, by construction, $\phi^{m-1}(V_l^m)=l$ for every $l=1,\ldots,N$, and from the choice of $\phi^{m-1}$, it follows that $\lambda^{m,l}_l\geq \mu^{m,l}_l$ $P$-almost surely for every $l=1,\ldots,N$ and $m\in \mathbb{N}$.
    Hence, $\alpha_l=\lim_{m\to \infty} \lambda^{M_m,l}_l\geq \lim_{m\to \infty }\mu^{M_m,l}_l=\beta_l$ $P$-almost surely for every $l=1,\ldots,N$.
    This is possible only if $\alpha_l=\beta_l$ $P$-almost surely for every $l=1,\ldots,N$, showing that $f(Y)=Y$.
\end{proof}
\section{Applications}
\subsection{Fixed point theorem for sequentially closed and bounded sets in $\left(L^0\right)^d$}
\begin{proposition}
	Let $\mathcal{K}$ be an $L^0$-convex, sequentially closed and bounded subset of $(L^0)^d$ and 
	$f \colon \mathcal{K}\to \mathcal{K}$ a local,
	sequentially continuous function. Then $f$ has a fixed point. 
\end{proposition}
\begin{proof}
	Since $\mathcal{K}$ is bounded, there exists a conditional simplex $\mathcal{S}$ such that $\mathcal{K}\subseteq \mathcal{S}$. 
	Now define the function $h \colon \mathcal{S}\to \mathcal{K}$ by 
	\begin{align*}
		h(X)=\begin{cases} X, &\text{if } X \in \mathcal{K},\\
			\argmin\set{\norm{X-Y} \colon Y \in \mathcal{K}}, &\text{else}.
		\end{cases}
	\end{align*}
	This means, that $h$ is the identity on $\mathcal{K}$ and the projection on $\mathcal{K}$ for the elements in $\mathcal{S}\setminus \mathcal{K}$.
	Due to \citep[Corollary 5.5]{cheridito2012}  this minmium exists and is unique. Therefore $h$ is well-defined.

	We can characterize $h$ by 
	\begin{align}\label{char01}
		Y=h(X)\Leftrightarrow \langle X-Y,Z-Y \rangle \leq 0,\ \text{ for all } Z \in \mathcal{K}.
	\end{align}

	Indeed, let $\langle X-Y,Z-Y \rangle \leq 0$ for all $Z \in \mathcal{K}$. Then 
	\begin{multline*}
		\norm{X-Z}^2=\norm{(X-Y)+(Y-Z)}\\
		=\norm{X-Y}^2+2\langle X-Y,Y-Z\rangle +\norm{Y-Z}^2\geq \norm{X-Y}^2,
	\end{multline*}
	which shows  the minimizing property of $h$.
On the other hand, let $Y=h(X)$. Since $\mathcal{K}$ is convex, $\lambda Z+(1-\lambda) Y\in \mathcal{K}$ for any $\lambda \in (0,1](\mathcal{A})$ and $Z\in\mathcal{K}$. By standard calculation,
\begin{align*}
	\norm{X-(\lambda Z+(1-\lambda)Y)}^2\geq\norm{X-Y}^2
\end{align*}
yields
$0\geq -2\lambda \langle X,-Y\rangle +(2\lambda -\lambda^2)\langle Y,Y \rangle+ 2 \lambda \langle X, Z\rangle  -\lambda^2\norm{Z}^2-2\lambda(1-\lambda)\langle Z,Y \rangle$.
Dividing by $\lambda>0$ and  letting $\lambda \downarrow 0$ afterwards yields
\begin{equation*}
	0\geq -2 \langle X,-Y\rangle +2\langle Y,Y \rangle+ 2 \langle X, Z\rangle -2\langle Z,Y \rangle=2\langle X-Y,Z-Y\rangle,
\end{equation*}
which is the desired claim.

Furthermore, for any $X,Y \in \mathcal{S}$ holds 
\begin{equation*}
	\norm{h(X)-h(Y)}\leq \norm{X-Y}.
\end{equation*}
Indeed, 
\begin{equation*}
	X-Y=(h(X)-h(Y))+X-h(X)+h(Y)-Y=:(h(X)-h(Y))+c
\end{equation*}
which means
\begin{equation}\label{char02}
	\norm{X-Y}^2=\norm{h(X)-h(Y)}^2+\norm{c}^2+2\langle c,h(X)-h(Y)\rangle.
\end{equation}
Since
\begin{equation*}
	\langle c, h(X)-h(Y)\rangle=-\langle X-h(X),h(Y)-h(X)\rangle-\langle Y-h(Y),h(X)-h(Y) \rangle, 
\end{equation*}
by \eqref{char01}, it follows  that $\langle c, h(X)-h(Y)\rangle\geq 0$ and \eqref{char02} yields $\norm{X-Y}^2\geq \norm{h(X)-h(Y)}^2$. 
This shows that $h$ is sequentially continuous.

The function $f\circ h$ is a sequentially continuous function mapping from $\mathcal{S}$ to $\mathcal{K}\subseteq \mathcal{S}$.
Hence, there exists a fixed point $f\circ h(Z)=Z$. 
Since $f\circ h$ maps to $\mathcal{K}$, this $Z$ has to be in $\mathcal{K}$.
But then we know $h(Z)=Z$ and therefore $f(Z)=Z$ which ends the proof.
\end{proof}

\begin{remark}
	In \citet{cond01} the concept of conditional compactness is introduced and it is shown that there is an equivalence between conditional compactness 
	and conditional closed- and boundedness in $(L^0)^d$.
	In that context we can formulate the conditional Brouwer fixed point theorem as follows.
	A sequentially continuous function $f \colon \mathcal{K}\to \mathcal{K}$ such that $\mathcal{K}$ is a conditionally compact and $L^0$-convex subset of $(L^0)^d$ has a fixed point.
\end{remark}

\subsection{Applications in Conditional Analysis on $\left(L^0\right)^d$}
Working in $\mathbb{R}^d$ the Brouwer fixed point theorem can be used to prove several
topological properties and is even equivalent to some of them. 
In the theory of $(L^0)^d$ we will shown that the conditional Brouwer fixed point theorem has several implications as well.

Define the \emph{unit ball} in $(L^0)^d$ by $\mathcal{B}(d)=\set{X \in (L^0)^d \colon \norm{X}\leq 1}$. 
Then by the former theorem any local, sequentially continuous function $f \colon \mathcal{B}(d)\to \mathcal{B}(d)$ has a fixed point.
The \emph{unit sphere} $\mathcal{S}(d-1)$ is defined as $\mathcal{S}(d-1)=\set{X \in (L^0)^d \colon \norm{X}=1}$.

\begin{definition}
	Let $\mathcal{X}$ and $\mathcal{Y}$ be subsets of $(L^0)^d$.
	An \emph{$L^0$-homotopy} of two local, sequentially continuous functions $f,g \colon \mathcal{X}\to \mathcal{Y}$ is a jointly local, 
	sequentially continuous function
	$H \colon \mathcal{X}\times[0,1](\mathcal{A})\to \mathcal{Y}$ such that $H(X,0)=f(X)$ and $H(X,1)=g(X)$.
	Jointly local means $H(\sum_{j \in \mathbb{N}} 1_{A_j} X_j, \sum_{j \in \mathbb{N}} 1_{A_j} t_j)=
	\sum_{j \in \mathbb{N}} 1_{A_j} H(X_j,t_j)$ for any partition $(A_j)_{j\in\mathbb{N}}$, $(X_j)_{j\in\mathbb{N}}$ in $\mathcal{X}$ and 
	$(t_j)_{j\in\mathbb{N}}$ in $[0,1](\mathcal{A})$. 
	Sequential continuity of $H$ is therefore $H(X_n,t_n)\to H(X,t)$ whenever $X_n\to X$ and $t_n\to t$ both $P$-almost surely 
	for $X_n,X\in\mathcal{X}$ and $t_n,t\in[0,1](\mathcal{A})$.
\end{definition}

\begin{lemma}
	The identity function of the sphere is not $L^0$-homotop to a constant function.
\end{lemma}

The proof is a consequence of the following lemma.

\begin{lemma}
	There does not exist a local, sequentially continuous function $f \colon \mathcal{B}(d)\to \mathcal{S}(d-1)$ which is the identity on $\mathcal{S}(d-1)$.
\end{lemma}

\begin{proof}
	Suppose there is this local, sequentially continuous function $f$.
	Define $g \colon \mathcal{S}(d-1)\to \mathcal{S}(d-1)$ by $g(X)=-X$.
	Then the composition $g\circ f \colon \mathcal{B}(d)\to \mathcal{B}(d)$, which actually maps to $\mathcal{S}(d-1)$, is local and sequentially continuous. 
	Therefore, this has a fixed point $Y$ which has to be in $\mathcal{S}(d-1)$, since this is the image of $g\circ f$.
	But we know $f(Y)=Y$ and $g(Y)=-Y$ and hence $g\circ f(Y)=-Y$. 
	Therefore, $Y$ cannot be a fixed point (since $0\notin \mathcal{S}(d-1)$) which is a contradiction.
\end{proof}

It directly follows that the identity on the sphere is not $L^0$-homotop to a constant function.
In the case $d=1$ we get the following result which is  the $L^0$-version of an $L^0$-intermediate value theorem.

\begin{lemma}
	Let $X,\overline{X} \in L^0$ with $X\leq \overline{X}$. Let $\left[X,\overline{X}\right]=\left\{Z\in L^0 \colon X\leq Z\leq \overline{X}\right\}$ and 
	$f \colon \left[X,\overline{X}\right]\to L^0$ be a local, sequentially continuous function. 
	Define $A=\set{f(X)\leq f\left(\overline{X}\right)}$.
	Then for every  
	$Y \in \left[1_A f(X)+1_{A^c} f\left(\overline{X}\right),1_A f\left(\overline{X}\right)+1_{A^c}f(X)\right]$ there exists $\overline{Y} \in \left[X,\overline{X}\right]$ with $f\left(\overline{Y}\right)=Y$.
\end{lemma}

\begin{proof}
    Since $f$ is local, it is sufficient to prove the case for $f(X)\leq f\left(\overline{X}\right)$ which is $A=\Omega$. For the general case we would consider $A$ and $A^c$ separately,
    obtain $1_Af\left(\overline{Y}_1\right)=1_AY$, $1_{A^c}f\left(\overline{Y}_2\right)=1_{A^c}Y$ and by locality we have $f\left(1_A \overline{Y}_1+1_{A^c}\overline{Y}_2\right)=Y$. So suppose $Y\in \left[f\left(X\right),f\left(\overline{X}\right)\right]$ in the rest of the proof.

    Let first $f(X)<Y<f\left(\overline{X}\right)$. 
    Define the function $g \colon \left[X,\overline{X}\right] \to \left[X,\overline{X}\right]$ by  
    \begin{equation*}
        g(V):=p(V-f(V)+Y) \quad \text{with} \quad p(Z)=1_{\set{Z\leq \vphantom{\overline{X}}X}}X+1_{\set{X\leq Z \leq \overline{X}}}Z+1_{\set{\overline{X}\leq Z}}\overline{X}.
    \end{equation*}
    Notice that as a sum, product, and composition of local, sequentially continuous functions, $g$ is so as well.
    Hence, $g$ has a fixed point $\overline{Y}$. 
    If $\overline{Y}=X$, it must hold $X-f(X)+Y \leq X$ which means $Y\leq f(X)$ which is a contradiction.
    If $\overline{Y}=\overline{X}$, it follows $f\left(\overline{X}\right)\leq Y$, which is also a contradiction.
    Hence, $\overline{Y}=\overline{Y}-f\left(\overline{Y}\right)+Y$ which means $f\left(\overline{Y}\right)=Y$.

    If $Y=f(X)$ on $B$ and  $Y=f\left(\overline{X}\right)$ on  $C$, it holds that $f(X)<Y<f\left(\overline{X}\right)$ on $(B\cup C)^c=:D$.
    Then we find  $\overline{Y}$ such that $f\left(\overline{Y}\right)=Y$ on $D$.
    In total $f\left(1_B X+1_{C\setminus B}  \overline{X}+1_D \overline{Y}\right)=1_B f(X)+1_{C\setminus B}f\left(\overline{X}\right)+1_D f\left(\overline{Y}\right)=Y$.
    This shows the claim for general $Y \in \left[f(X),f(\overline{X})\right]$.
\end{proof}

\bibliographystyle{econometrica}
\bibliography{bibfix}
\end{document}